\documentclass[12pt,A4]{article}

\DeclareSymbolFont{AMSb}{U}{msb}{m}{n}
\DeclareSymbolFontAlphabet{\Bbb}{AMSb}

\usepackage[ansinew]{inputenc} 
\usepackage{latexsym}

\usepackage{amstext,amsfonts,amsmath,amsthm,graphicx,amssymb,amscd,epsfig}
\usepackage{rotating}

\newtheorem{theorem}{Theorem}[section]

\newtheorem{lemma}[theorem]{Lemma}

\newtheorem{proposition}[theorem]{Proposition}

\newcommand{\abs}[1]{\left\vert#1\right\vert}

\newcommand{\C}{\mathbb{C}}
\newcommand{\R}{\mathbb{R}}
\newcommand{\Z}{\mathbb{Z}}
\newcommand{\N}{\mathbb{N}}
\newcommand{\Q}{\mathbb{Q}}
\newcommand{\T}{\mathbb{T}}

\title{\bf Rotation numbers for planar attractors of equivariant homeomorphisms}
\author{B. Alarc\'on\thanks{Work sopported by grant
EX2009-0283 of the Ministerio de Educaci\'on, grants MICINN-08-MTM2008-06065 and MICINN-12-MTM2011-22956 of the Ministerio de Ciencia e 
Innovaci\'on and by European Regional Development Fund through the programme
COMPETE and from the Portuguese Government through the Fundac\~ao para a Ci\^encia e a Tecnologia (FCT) under the project
PEst-C/MAT/UI0144/2011.}}

\date{{}}

\begin{document}
\maketitle

\vspace{-0.5cm}

{\footnotesize
\centerline{University of Oviedo} \centerline{C/ Calvo Sotelo,
s/n} \centerline{PC: 33007 Oviedo; SPAIN} \centerline{E-mail
address: alarconbegona@uniovi.es} }
\medskip

\begin{abstract}

Given an integer $m>1$ we consider $\Z_m-$~equivariant and orientation preserving
homeomorphisms in $\R^2$ with an asymptotically stable fixed
point at the origin. We present examples without periodic points and having some complicated dynamical features.
The key is a preliminary construction of  $\Z_m-$~equivariant Denjoy maps of the circle.
\end{abstract}

\section{Introduction}

This work is motivated by the study of the global behavior of a planar map having a fixed
point which is asymptotically stable but is not a global attractor. In \cite{irrationalRotNumer}, the authors 
show that it can happen even when there is no periodic points different 
from the fixed point. Actually, they construct examples of planar dissipative homeomorphisms $f$ such that the set $Rec(f)\setminus \{p\}$ is a 
Cantor set with almost automorphic dynamics, being $Rec(f)$ the set of recurrent points of $f$ and $p$ 
the asymptotically stable fixed point. Besides, they show that this behaviour is strongly related to the fact that these 
planar attractors have irrational rotation number.\par

As we are interested in systems with symmetry, in this work we construct symmetric and orientation preserving
dissipative homeomorphisms in the plane with an asymptotically
stable fixed point and irrational rotation number. These examples have no periodic points different from the fixed point and 
also present a complicated dynamical features. This constructions is based on the dissipative homeomorphisms given in
\cite{irrationalRotNumer},\par 

Let $f:\R^2 \rightarrow \R^2$ be an orientation preserving
homeomorphism with a fixed point which is not a global attractor
and its basin of attraction is unbounded. In that case, Theory of
Prime Ends due to Carath\'eodory is applied and $f$ induces an
orientation preserving homeomorphism $f^{*}$ in the space of prime
ends. Since this space is homeomorphic to the circle, it is
possible to associate a rotation number to $f$ being the rotation
number of $f^{*}$. \par
\medskip
The authors prove in \cite{irrationalRotNumer} that if $f$ is
dissipative and has an irrational rotation number, then the
induced map in the space of prime ends is always conjugated to a
Denjoy map. Otherwise, they prove in \cite{rationalRotNumer} that
periodic orbits different from fixed points can appear when the
rotation number is rational.
\medskip

Given a Lie group $\Gamma$ acting on $\R^2$, a map $f:\R^2 \rightarrow \R^2$ is
said to be $\Gamma-$\textit{equivariant} (or a $\Gamma-$symmetric
map) if for all $x\in \R^2$ and $\sigma \in \Gamma$

$$f(\sigma x)=\sigma f(x).$$

In \cite{ACL1} the authors prove that if the group $\Gamma$ is
$SO(2)$ or contains any reflection, the local dynamics of the
asymptotically stable fixed point implies global dynamics. That work
shows that the symmetry forces the global attraction to arise from
an asymptotically stable fixed point in all cases except when the map is
$\Z_m-$equivariant.
\medskip

In \cite{ACL2} are constructed $\Z_m-$equivariant homeomorphisms
of the plane  with an asymptotically stable fixed point and having
periodic points of period $m$ and rotation number $1/m$. So we
might be led to think that the presence of the $\Z_m-$~symmetry
implies that the rotation number of a homeomorphism should be
rational. In this article we give examples which show that this is
false. We prove the existence of $\Z_m-$~equivariant
homeomorphisms with an asymptotically stable fixed point such that the induced map in the space of prime
ends is conjugated to a Denjoy map, which is also
$\Z_m-$~equivariant. The idea is to reproduce the construction
given in \cite{irrationalRotNumer} in the context of symmetry.
\medskip

This work is organized as follows: In Section \ref{secNotation} we
explain some notations and results of Denjoy maps in the circle
that will be used. In Section \ref{secZm} we explain the problem
in the context of symmetry and construct $\Z_m-$equivariant Denjoy
maps in the circle. In section \ref{secZnR2} we prove the
existence of homeomorphisms of the plane which induce a symmetric
Denjoy map in the space of prime ends with the help of some results in
\cite{irrationalRotNumer} and Section \ref{secZm}.

%

\section{Notation and Denjoy map in the circle} \label{secNotation}

We introduce the same notation as in \cite{irrationalRotNumer}:
\medskip

We consider the quotient space $\T=\R/\Z$ and points $\bar
\theta=\theta + \Z$, with $\theta \in \R$. Although all figures
are sketched  on the unit circle ${\mathbb{S}}^1=\{z\in \C :
\abs{z}=1 \}$, which is homeomorphic to $\T$.
\medskip

The \textit{distance between two points}  $\bar \theta_1, \bar
\theta_2 \in \T$ is $$dist_\T(\bar \theta_1, \bar
\theta_2)=dist_\R (\theta_1-\theta_2, \Z)$$ where $dist_\R$
indicates the distance from a point to a set on the real line.
\medskip

A \textit{closed counter-clockwise arc} in $\T$ from $p$ to $q$,
$p\neq q$, will be denoted by $\alpha=\widehat{pq}$ and by $\dot
\alpha$ its corresponding \textit{open arc}.
\medskip

We define the \textit{cyclic order} as follows: Given three
different points $p_0, p_1, p_2 \in {\mathbb{T}}$ we say that $p_0
\prec p_1 \prec p_2$  if $p_1 \in \widehat{p_0p_2}$.
\medskip

We define the \textit{cyclic order for arcs} as follows: Given
three pairwise-disjoint arcs $\alpha_0, \alpha_1, \alpha_2 \subset
{\mathbb{T}}$, we say that $\alpha_0 \prec \alpha_1 \prec
\alpha_2$ if $p_0\prec p_1\prec p_2$ for some $\;p_0 \in \alpha_0,
\;p_1 \in \alpha_1, \;p_2 \in \alpha_2$.
\medskip

A \textit{Cantor set} $C$ is any compact totally disconnected
perfect subset of $\T$ (see \cite{topo}). $C$ may be considered as
$$C=\T \setminus \bigcup_{k=0}^{\infty} \; \dot{\alpha}_{k},$$
where $\{\alpha_k\}_{k\geq 0}$ is a family of pairwise disjoint
closed arcs in $\T$. The set of \textit{accessible} and
\textit{inaccessible points} will be denoted by $A$ and $I$,
respectively. The set $A$  is composed by the end points of all
$\alpha_k$, thus $$I=\T\setminus \bigcup_{k=0}^{\infty} \;
\alpha_{k}.$$
\medskip

Using $C$ we define an equivalence relation on $\T$ by putting
$\bar \theta_1 \sim \bar \theta_2$ (mod $C$) if  $\bar
\theta_1=\bar \theta_2 \quad \text{or} \quad \bar \theta_1, \bar
\theta_2 \in \alpha_k \quad \text{for some} \quad k\geq 0$. Thus,\textit{}
the \textit{Cantor function} associated to $C$ is a continuous
function $\mathcal{P}: \T \to \T$ such that
$$\mathcal{P}(\bar \theta_1)= \mathcal{P}(\bar \theta_2)\; \Leftrightarrow \; \bar \theta_1 \sim \bar \theta_2. $$

The intuitive idea of this type of maps is to collapse every arc
$\alpha_k$ into a point in such a way that the cyclic order is preserved.
See \cite{irrationalRotNumer} and \cite{robinson} for more details.



The Cantor function $\mathcal P$ is onto and $\mathcal{P}(A)$ is a
countable and dense subset of $\T$. See \cite{Markley} for more
details.

\medskip

The \textit{rotation} $R_{\bar \eta}:\T \to \T$ is defined by
$R_{\bar \eta}(\bar \theta)=\overline{\theta + \eta}$, where $\eta,
\theta \in \R$. Given a homeomorphism $f$ of $\T$, the \textit{$f-$~orbit} starting at a point $\bar\theta \in \T$ is 
denoted by $\mathcal{O}(\bar\theta)$. The \textit{$\omega-$~limit} of a point $\bar\theta \in \T$ is denoted by $\omega(\bar\theta)$. It
is well known (see \cite{robinson}) that an orientation preserving
homeomorphism $f$ in $\T$ with rational rotation number has
periodic points. However, if $f$ has irrational rotation number
$\rho(f)=\bar\tau \notin \Q$, then:

\begin{itemize}
\item[(a)] $\omega(\bar\theta)$ is independent of $\bar\theta$.
\item[(b)] $f$ is semi-conjugate to the rigid rotation map
$R_{\bar \tau}$. The semi-conjugacy takes the orbits of $f$ to orbits
of $R_{\bar \tau}$, is at most two to one on $\omega(\bar\theta)$ and
preserves orientation. \item[(c)] If $\omega(\bar\theta)=\T$, then
$f$ is conjugate to $R_{\bar \tau}$ and the minimal set of $f$ is the
whole circle $\T$. \item[(d)] If $\omega(\bar\theta)\neq\T$, then
the semi-conjugacy from $f$ to $R_{\bar \tau}$ collapses the closure of
each open interval in the complement of $\omega(\bar\theta)$ to a
point. Moreover the only minimal set of $f$ is a Cantor set $C$ in
$\T$.

\end{itemize}

An orientation preserving homeomorphism $f: \T \to \T$ is said to
be a Denjoy map if $f$ has an irrational rotation number $\bar
\tau$ and $f$ is not conjugated to any rotation. In that case, $f$
admits a Cantor minimal set $C_f$ that attracts all orbits in the
future and in the past ---and every point in $C_f$ is a recurrent
point of $f$. We can associate to $C_f$ a Cantor function $\cal P$
which is unique up to rotations and such that $\mathcal{P} \circ f
= R_{\bar \tau} \circ \mathcal{P}$. So $\mathcal{P}$ is a
semi-conjugacy from $f$ to $R_{\bar \tau}$.



\medskip

The construction in \cite{robinson} of a Denjoy map with
irrational rotation number $\bar \tau$ consists of choosing a
point $\bar \theta \in \T$ and determining a family of pairwise
disjoint open arcs in $\T$ with decreasing lengths whose sum is
one and the complement of the union of all of them is a Cantor
set. Each arc is identified with an element of the orbit of $\bar
\theta$ via $R_{\bar \tau}$ which is always dense in $\T$. They also
are put in the same order as the elements of the orbit, that is,
preserving the cyclic order. These open intervals correspond to
the gaps of the Cantor set and the union of the two extremes of
all the intervals is the accessible set $A$ of $C_f$. Next step is
to define $f$ on the union of the intervals and then extend the
map to the closure.
\medskip

But it is also possible to generate a Denjoy map considering the $R_{\bar \tau}-$orbit
orbit of more than one point. Markley proved in \cite{Markley}
that given an irrational number $\tau \notin \Q$ and a countable
set $D\neq \emptyset$ in $\T$ such that $R_{\bar \tau} D=D$, there
exists a Denjoy map with rotation number $\bar \tau$ and minimal
Cantor set with Cantor function verifying $\mathcal{P}(A)=D$ and
being unique up to rotations. For instance, the construction in
\cite{robinson} corresponds to the countable set $\mathcal{P}(A)$
composed by the orbit of a unique point $\overline{\varphi}\in \T$
by the rotation $R_{\bar \tau}$, say
$$\mathcal{P}(A)=\{\overline{\varphi+ n\tau}: n\in \Z\}$$

Figure \ref{figuracantorRafa} illustrates the construction of a
Denjoy map considering the orbit of two different points, which is
well explained in \cite{irrationalRotNumer}. The corresponding
countable set is  $$\mathcal{P}(A)=\{\overline{\varphi + n\tau}:
n\in \Z\} \cup \{\overline{\psi + n\tau}: n\in \Z\}.$$

\begin{figure}[hh]
\centering
\includegraphics[scale=0.4]{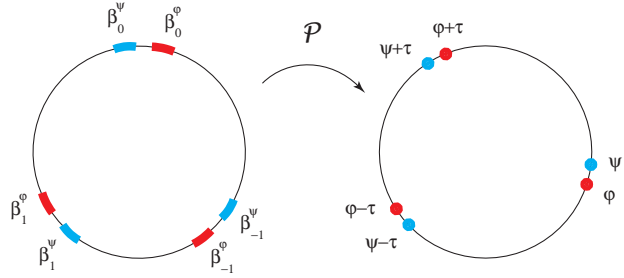}
\caption{Construction of a Cantor set with two orbits.}  \label{figuracantorRafa}
\end{figure}

%

\section{$\Z_m-$equivariant Denjoy maps in the circle} \label{secZm}

Observe that the construction in Figure \ref{figuracantorRafa}
depends on the points $\bar \varphi, \bar \psi \in \T$. Since it
can be made with every two points, in this section we consider
$\bar \psi= R_{\frac{1}{2}}\bar \varphi$ in order to look for any symmetry
of the Denjoy map (see Figure \ref{figuracantorZ2}). This
motivated us to study the more general case when the countable set
$D$ is the union of the orbits of points which are the rational
rotation $R_{\frac{k}{m}}$ of a given point $\bar\varphi \in
\T$ for same $k=0,...,m-1$. That is, given a point $\bar \varphi
\in \T$ and numbers $\tau \notin \Q$, $m\in \N$ we consider the
set
$$\mathcal{P}(A)=\bigcup_{k=0}^{m-1} \{\overline{\varphi^k+n\tau}:
n\in \Z\}$$ where   $\bar \varphi^k=R_{\frac{k}{m}} \bar
\varphi$,  for $k=0,...,m-1$. See Figure \ref{figuracantorZm}.
\medskip
\begin{figure}[hh]
\centering
\includegraphics[scale=0.5]{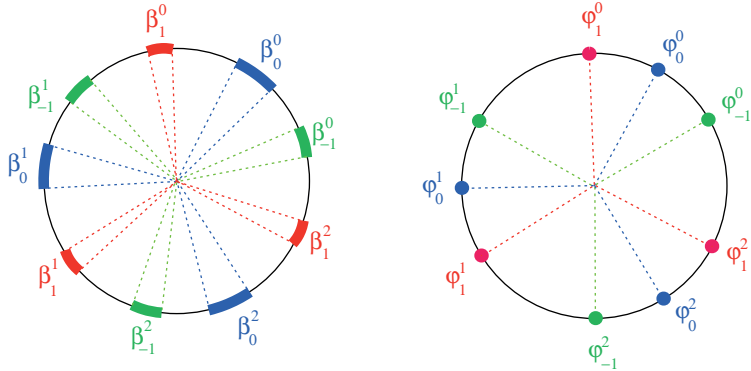}
\caption{Construction of a Cantor set with $3$ symmetric points.}
\label{figuracantorZm}
\end{figure}



Since $\Z_m=\{R_{\frac{k}{m}}\}_{k=0}^{m-1}$, given a point
$\bar \varphi \in \T$ we define the set $\{\bar
\varphi^k\}_{k=0}^{m-1}$ as the\textit{ orbit of the group} $\Z_m$
(or $\Z_m-$orbit) of $\bar \varphi$, where $\bar \varphi
^k=R_{\frac{k}{m}} \bar \varphi$.


\medskip

In addition, $\Z_m$ is a cyclic compact Lie group generated by
$R_{\frac{1}{m}}$. So, in order to stay that a map $f$ is
$\Z_m-$equivariant, we only need to prove that
$$f(\bar\varphi+\frac{1}{m})=f(\bar \varphi)+ \frac{1}{m}.$$ See
\cite{golu2} for more details.
\medskip

In this section we prove the existence of $\Z_m-$equivariant
Denjoy maps. For simplicity, details of the proof will be explained only in case $m=2$ because the case $m>2$
is analogous. Firstly we construct a Cantor set which
is invariant under the rotation $R_{\frac{1}{2}}$. Secondly we prove the
existence of $\Z_2-$equivariant Denjoy maps in the circle with
the constructed Cantor set as its minimal set. Finally, we give
the keys of the proof in case $m>2$.
\medskip


\begin{lemma}[Herman {\cite{Herman}, p. 140}] \label{lemHerman} Let $D_1$, $D_2$ be two
dense subsets in $\R$, and $\phi:D_1 \to D_2$ be a strictly
increasing map which is onto. Then $\phi$ can be extended to a
monotone strictly increasing continuous map from $\R$ to $\R$.
\end{lemma}

\begin{lemma} \label{propCantor2} Let  $\tau \notin \Q$. Given a point $\bar \varphi$
in the circle, there exists a
Cantor set $C$ such that $R_{\frac{1}{2}} C=C$ and the associated Cantor
function $\mathcal{P}:\T \to \T$ verifies:

\begin{itemize} \item[(a)] $\mathcal{P}(A)=\{\overline{\varphi+n \tau} / n\in \Z \} \cup
\{\overline{\varphi'+n \tau} / n\in \Z \} \subset \T$, where $A$ is the
accessible set of $C$ and $\overline{\varphi}'=R_{\frac{1}{2}}
\overline{\varphi}$

\item[(b)] $\mathcal{P}$ is $\Z_2-$equivariant.
\end{itemize}
\end{lemma}

\begin{proof}

Given an angle $\tau \notin \Q$ and an orbit
$\{\overline{\varphi+n \tau} / n\in \Z\}\subset \T$ we consider
the countable dense set
$$\mathcal{D}=\{\overline{\varphi+n \tau} / n\in \Z \} \cup
\{\overline{\varphi'+n \tau} / n\in \Z \} \subset \T,$$ such that
$\bar\varphi'=R_{\frac{1}{2}} \bar \varphi$.

%

\begin{figure}[hh]
\centering
\includegraphics[scale=0.5]{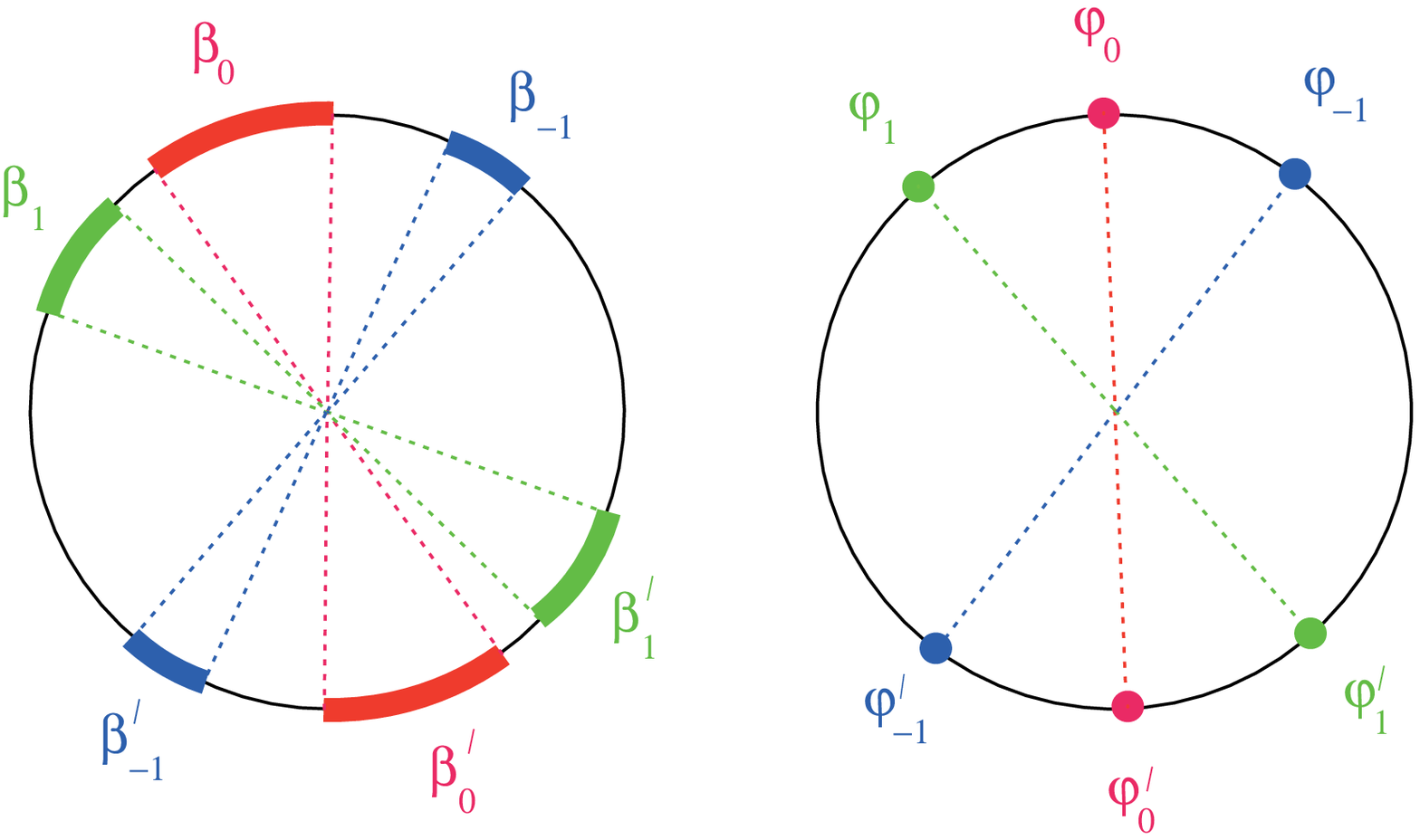}
\caption{Construction of a Cantor set with two symmetric points.}    \label{figuracantorZ2}
\end{figure}

We claim that there exists a family $\{A_n\}_{n\in \N}$ of compact
subsets of $\T$ such that $$\bigcap _{n\in \N} A_n= \T \setminus
\bigcup_{n\in \N}(\dot{\beta}_n \cup \dot{\beta'}_{n}
\cup\dot{\beta}_{-n}\cup\dot{\beta'}_{-n}),$$ where $\{\beta'_n,
\beta_n\}_{n\in \Z}$ is a family of pairwise disjoint closed arcs
in $\T$ such that $\beta'_n= R_{\frac{1}{2}} \beta_n$. See Figure
\ref{figuracantorZ2}.
\medskip

We construct the family $\{A_n\}_{n\in \N}$ by induction on $n\in
\N$. It verifies:
\medskip

\begin{itemize}
\item[(a)] $A_n=\bigcup_{i=1}^ {2k(n)} \gamma_i^ n $, where
$\gamma_i^ n=\widehat{a_i^ n b_i^ n}$ is a closed arc in $\T$ with
end points $a_i^n, b_i^n$. In addition $\gamma_i^ n \cap \gamma_j^
n=\emptyset$ for all $i\neq j$ with $i,j=1,..., 2k(n)$, where
$k(n)=2n-1$. Moreover $\gamma_{k(n)+j}^ n=R_{\frac{1}{2}} \gamma_j^ n$ for
$j=1,...,k(n)$.

\item[(b)] $\bigcup_{j=1}^n\{a_i^ {j-1}, b_i^
{j-1}\}_{i=1}^{2k(j-1)} \subset A_n$. That is, every extreme point
of each arc $\gamma_i^{k}$ of $A_k$  belongs to $A_n$, for
$k=1,..,n-1$.

\item[(c)] $A_{n} \subset A_{n-1}$.

\item[(d)] $R_{\frac{1}{2}} A_n=A_n$.

\item[(e)] The correspondence which associates to each point $\bar
\varphi_N=\overline{\varphi+N\tau}$ the arc $\beta_N$, for all
$\abs{N}\leq n-1$, preserves the cycle order. Equivalently, if
$\bar \varphi_{N_1}\prec \bar \varphi_{N_2}\prec \bar
\varphi_{N_3}$, then $\beta_{N_1}\prec \beta_{N_2}\prec
\beta_{N_3}$  for all $\abs{N_i}\leq n-1$.
\end{itemize}
\medskip

Consider $A_0=\T$. We associate to the points
$\bar\varphi_0=\overline{\varphi+0\tau}$ and $\bar
\varphi'_0=\overline{\varphi'+0\tau}$ two open arcs $\beta_0$ and
$\beta'_0=R_{\frac{1}{2}}\beta_0$, respectively, preserving the cyclic
order such that $\beta_0\cap \beta'_0 =\emptyset $ and
$\mu(\T\setminus \beta_0 \cup \beta'_0) < \frac{1}{2}
\mu(A_0)=\frac{1}{2}.$ We define
$$A_1=A_0 \setminus \beta_0 \cup \beta'_0=\gamma_1^1 \cup
\gamma_2^1,$$ where  $\gamma_1^1, \gamma_2^1$ are closed,
$\gamma_1^1 \cap \gamma_2^1=\emptyset$ and $\gamma_2^1 = R_{\frac{1}{2}}
\gamma_1^1$. Clearly, $R_{\frac{1}{2}} A_1=A_1$ and $A_1 \subset A_0$. See
Figure \ref{figuraA0A1}.

\begin{figure}[hh]
\centering
\includegraphics[scale=0.3]{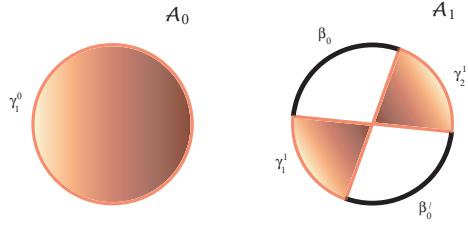}
\caption{The subsets $A_0, A_1$}    \label{figuraA0A1}
\end{figure}

\medskip

Now we associate to the points $\bar
\varphi_1=\overline{\varphi+\tau}$, $\bar
\varphi_{-1}=\overline{\varphi-\tau}$,
$\varphi'_1=\overline{\varphi'+\tau}$ and
$\varphi'_{-1}=\overline{\varphi'-\tau}$ four open arcs in $A_1$,
say $\beta_1$, $\beta_{-1}$,  $\beta'_1=R_{\frac{1}{2}}\beta_1$ and
$\beta'_{-1}=R_{\frac{1}{2}}\beta_{-1}$, respectively, preserving the
cycle order and being pairwise-disjoint such that
$$\mu(A_1\setminus \beta_1 \cup \beta'_1 \cup \beta_{-1} \cup
\beta'_{-1}) \leq \frac{1}{2} \mu(A_1).$$ We define
$$A_2=A_1\setminus \beta_1 \cup \beta'_1 \cup \beta_{-1} \cup
\beta'_{-1}= \gamma_1^2 \cup .. \cup \gamma_6^2, $$ where
$\gamma_1^2,..., \gamma_6^2$ are closed, pairwise-disjoint and
$\gamma_4^2 = R_{\frac{1}{2}} \gamma_1^2, \gamma_5^2 = R_{\frac{1}{2}} \gamma_2^2$
and $\gamma_6^2 = R_{\frac{1}{2}} \gamma_3^2$. Clearly, $R_{\frac{1}{2}} A_2=A_2$
and $\{a_1^1, b_1^1, a_2^1, b_2^1\}\subset A_2 \subset A_1$. See
Figure \ref{figuraA2}.

\begin{figure}[hh]
\centering
\includegraphics[scale=0.5]{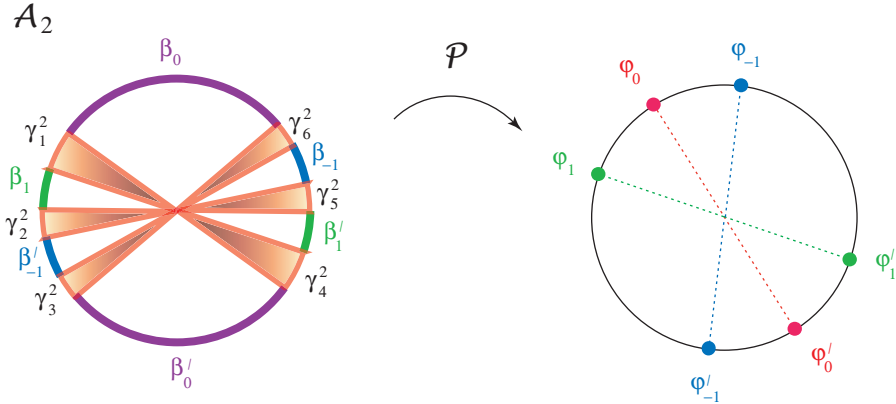}
\caption{The subset $A_2$}    \label{figuraA2}
\end{figure}

Suppose by induction that there exists $A_n$ verifying assumptions
$(a)$, $(b)$, $(c)$, $(d)$ and $(e)$.
\medskip

As $\tau \notin \Q$, the orbit of $\varphi$ is dense in $\T$ so we
can associate to the points
$\bar\varphi_n=\overline{\varphi+n\tau}$,
$\bar\varphi_{-n}=\overline{\varphi-n\tau}$,
$\bar\varphi'_n=\overline{\varphi'+n\tau}$ and
$\bar\varphi'_{-n}=\overline{\varphi'-n\tau}$ four open arcs in
$A_{n-1}$, say $\beta_n$, $\beta_{-n}$,  $\beta'_n$ and
$\beta'_{-n}$, respectively, such that

\begin{itemize}

 \item they are pairwise-disjoint and preserving the cyclic order,
 \item $\beta_{n}'=R_{\frac{1}{2}}\beta_n$ and $\beta_{-n}'=R_{\frac{1}{2}}\beta_{-n}$,
 \item if $\bar \varphi_{N_1}\prec
\bar \varphi_{N_2}\prec \bar \varphi_{N_3}$, then
$\beta_{N_1}\prec \beta_{N_2}\prec \beta_{N_3}$  for all
$\abs{N_i}\leq n$.
\end{itemize}
\medskip

In addition, the arcs verify $$\mu(\gamma_i^n\setminus \beta) \leq
\frac{1}{2}\mu(\gamma_i^n)$$ where $\beta$ can be one arc $\beta_n$,
$\beta_{-n}$,  $\beta'_n$ and $\beta'_{-n}$ or the union of two of
such arcs. The arc $\gamma_i^n$ is the component of $A_n$
containing  each $\beta$.

\medskip

Figures \ref{figuraA3C1} and  \ref{figuraA3C2} show how one component of $A_n$ could be cut by two arcs or by one. But in both cases the obtained new arcs should have less measure than the first one. That is,

\begin{itemize} \item[Case I:] Suppose the component $\gamma_i^ n$ is cut by only one arc $\beta_n$ (respectively $\beta'_{-n}$), then there appear two new closed arcs $\gamma_j^{n+1}, \gamma_{j+1}^{n+1}$ in $\gamma_i^n$ such that $$\mu(\gamma_j^{n+1} \cup \gamma_{j+i}^{n+1}) \leq \frac{1}{2}\mu(\gamma_i^n).$$
\item[Case II:] Suppose now  the component $\gamma_i^ n$ is cut by the two arcs $\beta_n$ and $\beta'_{-n}$. Then there appear three closed arcs $\gamma_j^{n+1}, \gamma_{j+1}^{n+1}, \gamma_{j+2}^{n+1}$ in $\gamma_i^n$ such that $$\mu(\gamma_j^{n+1} \cup \gamma_{j+i}^{n+1} \cup \gamma_{j+1}^{n+1}) \leq \frac{1}{2}\mu(\gamma_i^n).$$
\end{itemize}

\medskip

Observe that in the first case the arc $\gamma_{i}^{'n}=R_{\frac{1}{2}} \gamma_i^n$  has also been cut by the arcs $\beta'_n$ and $\beta_{-n}$ such a way the obtained new arcs are $\gamma_j^{'n+1}=R_{\frac{1}{2}}\gamma_j^{n+1}, \gamma_{j+1}^{'n+1}=R_{\frac{1}{2}}\gamma_{j+1}^{n+1}$ and $\gamma_{j+2}^{'n+1}=R_{\frac{1}{2}}\gamma_{j+2}^{n+1}$ all in $\gamma_i^{'n}$. The case that the arc $\gamma_i^n$ had been cut only by one arc $\beta$ is analogous.

\medskip
We define $$A_{n+1}=A_n\setminus \beta_n \cup \beta'_n \cup
\beta_{-n} \cup \beta'_{-n}= \bigcup_{i=1}^{2k(n+1)}
\gamma_i^{n+1} , $$ where $k(n)=2n-1$ and
$\{\gamma_i^{n+1}\}_{i=1}^{2k(n+1)}$ is a family of
pairwise-disjoint and closed arcs in $\T$ such that
$\gamma_{k(m+1)+i}^{n+1} = R_{\frac{1}{2}} \gamma_i^{n+1}$ for all
$i=1,..., k(n+1)$. Clearly, $R_{\frac{1}{2}} A_{n+1}=A_{n+1}$ and
$$\{a_i^1, b_i^1\}_{i=1}^{2k(1)} \cup... \cup \{a_i^n,
b_i^n\}_{i=1}^{2k(n)} \subset A_{n+1} \subset A_n\subset ...
\subset A_1,$$ that is, the two extreme points of each arc
$\gamma_i^k$ of $A_k$ belong to $A_{n+1}$ for all $k=1,...,n$.
Therefore, the family $\{A_n\}$ is well defined and $$\bigcap
_{n\in \N} A_n= \T \setminus \bigcup_{n\in \N}(\dot{\beta}_n \cup
\dot{\beta'}_{n} \cup\dot{\beta}_{-n}\cup\dot{\beta'}_{-n})$$
\medskip

\begin{figure}[hh]
\centering
\includegraphics[scale=0.5]{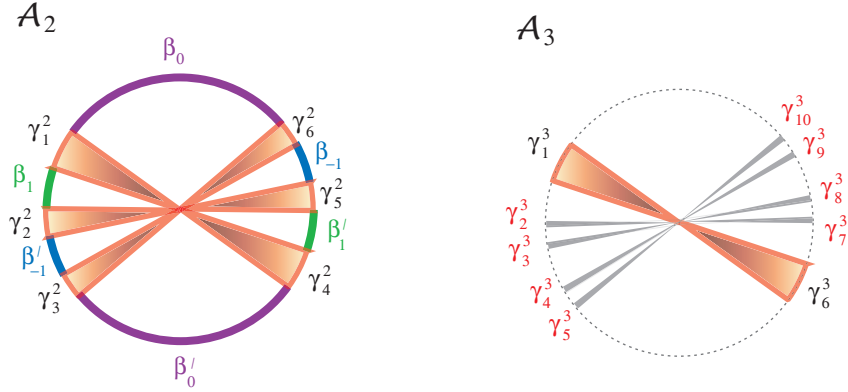}
\caption{The subset $A_3$. Case I}    \label{figuraA3C1}
\end{figure}

Now we claim that $\mu (\bigcap_{n=0}^{\infty} A_n)=0$, where
$\mu$ denotes the Lebesgue measure. By construction, for each
$m\geq 0$ there exists a natural number $\sigma(m)\geq m$ such
that $\mu(A_{\sigma(m)}) <\frac{1}{2}\mu(A_m)$. The compact set
$A_{\sigma(m)}$ is the $\sigma(m)-$th element of the family
$\{A_n\}_{n=0}^\infty$ such that every component of $A_m$ have
been cut by an arc $\beta_k$ or $\beta'_k$ with $k\in \Z$ at least
once.
\medskip

Let us consider the subfamily $\{A_{\sigma_{n}(m)}\}_{n=0}^\infty$
where the compact set $A_{\sigma_{n+1}(m)}$ is an element of
$\{A_n\}_{n=0}^\infty$ such that every component of
$A_{\sigma_{n}(m)}$ have been cut by an arc $\beta_k$ or $\beta'_k$
with $k\in \Z$ at least once. Then, $$\mu(A_{\sigma_{n+1}(m)})\leq
\frac{1}{2} \mu(A_{\sigma_{n}(m)}), \quad \forall n\geq 0$$ and
$A_{\sigma_{n+1}(m)}\subseteq A_{\sigma_{n}(m)}$ for all $n\geq
0$. So $$0\leq \mu(A_{\sigma_{n+1}(m)})\leq \frac{1}{2^n}\mu(A_m)<
\frac{1}{2^{n+1}}$$ and $$0\leq \mu(\bigcap_{n=0}^\infty A_n)\leq
\mu(\bigcap_{n=0}^\infty A_{\sigma_{n}(m)})=\lim_{n\to \infty}
\mu(A_{\sigma_{n}(m)})=0.$$

\begin{figure}[hh]
\centering
\includegraphics[scale=0.5]{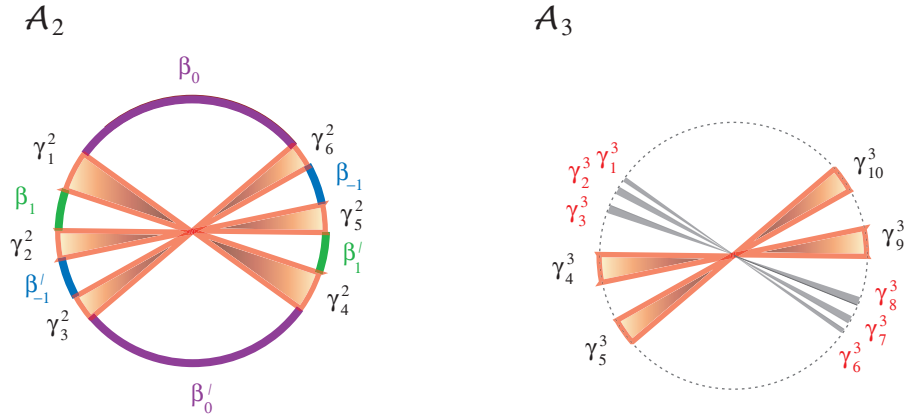}
\caption{The subset $A_3$. Case II}    \label{figuraA3C2}
\end{figure}

Now we claim that $$C=\bigcap _{n\in \N} A_n= \T \setminus
\bigcup_{n\in \N}(\dot{\beta}_n \cup \dot{\beta'}_{n}
\cup\dot{\beta}_{-n}\cup\dot{\beta'}_{-n})$$ is a Cantor set. $C$
is compact because it is an intersection of compact sets. Suppose
that there exists a connected component of $C$ different from a
singleton $\{p\}$, then there exists an arc belonging to $C$.
Therefore $\mu(\bigcap_{n\in \N} A_n) = 0$, which is
impossible. So $C$ is totally disconnected. Now  take a point
$x\in C$, then $x\in A_n$ for all $n\in \N$. In particular, for
each $n\in \N$, there exists a $i(n)-$th compact component of
$A_n$ such that $x\in \gamma_{i(n)}^n$. Let $x_n$ be an extreme of
$ \gamma_{i(n)}^n$ such that $x_n\neq x$. As $n$ increases the
component $ \gamma_{i(n)}^n$ is being cut in pieces such that the
measure of the complement of the gaps is less than the half of the
piece I have cut, so the distance between $x_n$ and $x$ is getting
smaller. That is, $ \forall \varepsilon >0$ there exists a $n_0\in
\N$ and a component $\gamma_{i(n)}^n$ such that, if $x$ is an
extreme of $\gamma_{i(n)}^n$,
 $$d(x_n,x)=L(\widehat{x_nx})= \mu(\gamma_{i(n)}^n)$$ where $L(\widehat{x_nx})$ denotes the length of the arc $\widehat{x_nx}$. Then, $$d(x_n,x)=\mu(\gamma_{i(n)}^n)<\frac{1}{2}\mu(\gamma_{i(n-1)}^{n-1})<...<\frac{1}{2^n}<\varepsilon, \quad \forall n>n_0.$$ And if $x\in \dot{\gamma}_{i(n)}^n$, then $$d(x_n,x)<d(x_n,y)<\frac{1}{2^n}<\varepsilon, \quad \forall n>n_0,$$ where $y$ is the other extreme of $\gamma_{i(n)}^n$. Then $C$ is perfect.
\medskip

Note that, by construction, $R_{\frac{1}{2}} C=C$ and the set of accessible
points of $C$ is the union of the two extremes of all components
$\gamma_i^n$ for all $n\in \N$.
\medskip

Let now to define the associated Cantor function $\mathcal{P}: \T
\to \T$ of $C$. Consider the function
$$\mathcal{P}_{*}:\bigcup_{n\in \Z}\dot{\beta}_n \cup
\dot{\beta'}_{n} \to \{\overline{\varphi+n \tau} : n\in \Z \} \cup
\{\overline{\varphi'+n \tau} : n\in \Z \}$$ such that
$\mathcal{P}_{*}(\dot{\beta}_n)=\overline{\varphi}_n$ and
$\mathcal{P}_{*}(\dot{\beta'}_n)=\overline{\varphi}'_n$, for all
$n\in \Z$. It is easy to verify by induction that it is well
defined considering the functions


$$\mathcal{P}_n: \bigcup_{i=-n}^{i=n} \dot{\beta}_i
\cup \dot{\beta'}_i \to \{\overline{\varphi +N\tau}:\abs{N}\leq
n\} \cup \{\overline{\varphi' +N\tau}:\abs{N}\leq n\}$$ such that
$\mathcal{P}_n(\dot{\beta}_N)=\overline{\varphi}_N$,
$\mathcal{P}_n(\dot{\beta'}_N)=\overline{\varphi}'_N$,
$\mathcal{P}_n(\dot{\beta}_{-N})=\overline{\varphi}_{-N}$,
$\mathcal{P}_n(\dot{\beta'}_{-N})=\overline{\varphi}'_{-N}$, for
all $\abs{N}\leq n$ and a given $n\in \N$.
\medskip

Observe that
$\mathcal{P}_{*}(\dot{\beta'}_n)=\mathcal{P}_{*}(\dot{\beta}_n)+\frac{1}{2}$
for all $n\in \Z$ by construction. This implies that
$\mathcal{P}_{*} \circ R_{\frac{1}{2}} = R_{\frac{1}{2}} \circ \mathcal{P}_{*}$ and
$\mathcal{P}_{*}$ is $\Z_2-$equivariant.
\medskip
Now we extend $\mathcal{P}_{*}$ to the function $\mathcal{P}$ at
the points in the closure of $\bigcup_{n\in \Z}\dot{\beta}_n \cup
\dot{\beta'}_{n}$ applying Lemma \ref{lemHerman}. 
\medskip

Actually, $\Omega=\bigcup_{n\in \Z}\dot{\beta}_n \cup
\dot{\beta'}_{n}$ is open and $\mu(\T \setminus \Omega)=0$, so
$\Omega$ is dense in $\T$. Note that $D=\{\overline{\varphi+n
\tau} : n\in \Z \} \cup \{\overline{\varphi'+n \tau} : n\in \Z \}$
is also dense because $\tau \in \R\setminus \Q$. Let consider
$\widehat{\Omega} \subset \R$ and $\widehat{D}\subset \R$ be the
lift of $\Omega$ and $D$, respectively. The function
$\widehat{\mathcal{P}}_{*}: \widehat{\Omega} \to \widehat{D}$ is
onto and preserves orientation, so it is strictly increasing.
By Lemma \ref{lemHerman}, $\widehat{\mathcal{P}}_{*}$ can be
extended to a continuous strictly increasing function
$\widehat{\mathcal{P}}:\R \to \R$ such that
$\widehat{\mathcal{P}}(x+1)=\widehat{\mathcal{P}}(x)+1$, for all
$x\in \R$. So $\widehat{\mathcal{P}}$ can be consider as a lift of
a cyclic order preserving and continuous function $\mathcal{P}:\T
\to \T$ such that $\mathcal{P} \circ R_{\frac{1}{2}} = R_{\frac{1}{2}} \circ
\mathcal{P}$ and $\mathcal{P}(A)=D_2$, where $A$ is the union of
the two extremes of every component $\gamma_i^n$ for all $n\in
\N$, the set of accessible points of the Cantor set $C$.
\end{proof}

\begin{proposition} \label{propDenjoy2} Let  $\tau \notin \Q$. Given a point $\bar \varphi \in \T$, there exists a
Denjoy map $f: \T \to \T$ such that:
\begin{itemize} \item[(a)] $f$ is $\Z_2-$equivariant. \item[(b)]
$\rho(f)=\bar \tau$. \item[(c)] $f$ has a minimal Cantor set as in
Lemma \ref{propCantor2}.
\end{itemize}
\end{proposition}

\begin{proof}

Let $\tau \notin \Q$ and consider a point $\bar \varphi \in \T$. By Lemma \ref{propCantor2},
we can construct a Cantor set
$C$ such that $R_{\frac{1}{2}} C=C$ and the associated Cantor function
$\mathcal{P}:\T \to \T$ is $Z_2-$equivariant.
\medskip

Let now to define the Denjoy homeomorphism with minimal set the
Cantor set $C$. We are going to apply Lemma \ref{lemHerman} in the
same way as the construction of $\mathcal{P}$ in Lemma
\ref{propCantor2}.
\medskip

Consider the bijection
$$f_{*}: \bigcup_{n\in \Z}\dot{\beta}_n \cup \dot{\beta'}_{n} \to
\bigcup_{n\in \Z}\dot{\beta}_n \cup \dot{\beta'}_{n}$$ verifying:

\begin{itemize}
 \item For each $n\in \Z$, $f_{*}(a_n)=a_{n+1}$, $f_{*}(b_n)=b_{n+1}$, $f_{*}(a'_n)=a'_{n+1}$ and $f_{*}(b'_n)=b'_{n+1}$. 
 \item $f_{*}\circ R_{\frac{1}{2}}=R_{\frac{1}{2}} \circ f_{*}$.
\end{itemize}

Analogously to the definition of the
Cantor function associated to $C$, it is easy to verify that
$f_{*}$ is well defined considering for each $n\in \N$ the
bijection $$f_{n}: \bigcup_{\abs{i}\leq n}\dot{\beta}_i \cup
\dot{\beta'}_{i} \to \bigcup_{\abs{i}\leq n+1}\dot{\beta}_i \cup
\dot{\beta'}_{i}$$ verifying:

\begin{itemize}
 \item $f_{n}(a_i)=a_{i+1}$, $f_{n}(b_i)=b_{i+1}$, 
$f_{n}(a'_i)=a'_{i+1}$ and $f_{n}(b'_i)=b'_{i+1}$.
 \item $f_{n}\circ R_{\frac{1}{2}}=R_{\frac{1}{2}} \circ f_{n}$.
\end{itemize}


Let $\Omega=\bigcup_{n\in \Z}\dot{\beta}_n \cup \dot{\beta'}_{n}$. Consider $widehat{\Omega}$ and 
$\widehat{f_{*}}: \widehat{\Omega} \to
\widehat{\Omega}$ being the lift of $\Omega$ and $f_{*}$, respectively. As $f_{*}$ is a
bijection and preserves the cyclic order, $\widehat{f_{*}}$ is a
strictly increasing bijection. Then, by Lemma \ref{lemHerman}, it
can be extended to a continuous orientation preserving bijection
$\widehat{f}: \R \to \R$ such that
$\widehat{f}(x+1)=\widehat{f}(x)+1$ for all $x\in \R$. Moreover
$\widehat{f}\circ R_{\frac{1}{2}}= R_{\frac{1}{2}} \circ \widehat{f}$ and, as every
continuous and one-to-one map in $\R$ is open, $\widehat{f}$ is an
orientation preserving homeomorphism in $\R$. Therefore
$\widehat{f}$ can be consider as a lift of an orientation
preserving $\Z_2-$equivariant homeomorphism $f:\T \to \T$.
\medskip

By construction, $R_{\bar \tau} \circ \mathcal{P}=\mathcal{P}
\circ f$. Therefore, $f$ is an orientation preserving
$\Z_2-$equivariant Denjoy homeomorphisms of the circle with
rotation number $\bar \tau$.
\medskip
\end{proof}

\begin{theorem} \label{teoDenjoyZm} Let $\tau \notin \Q$ and $m\in \N$. Given a point $\bar \varphi \in \T$, there exists a
$\Z_m-$equivariant Denjoy map with rotation number $\bar \tau$.
\end{theorem}

\begin{proof} The case $m=2$ is given in Proposition \ref{propDenjoy2}. The case $m\geq  3$ is analogous considering the dense and countable set $$D=\bigcup_{k=0}^{m-1} \{\overline{\varphi^k+n\tau}: n\in \Z\}$$

If $A_0= \T$, we may associate to
$\{\bar{ \varphi}_0^k \}_{k=0}^{m-1}$, the orbit of $\bar
\varphi_0=\bar \varphi$, the family of pairwise-disjoint open arcs
$\{\beta_0^k\}_{k=0}^{m-1}$ such that
$\mu(\bigcup_{k=0}^{m-1}\beta_0^{k})< \frac{1}{2}
\mu(A_0)=\frac{1}{2}$ and define $A_1=A_0\setminus
\bigcup_{k=0}^{m-1}\beta_0^{k}$.
\medskip

Analogously to case $m=2$ we can define by induction the sets
$$A_{n+1}=A_n\setminus \bigcup_{k=0}^{m-1} \beta_n^k \cup
\beta_{-n}^k, \quad \forall n\in \N$$ such that
\medskip

\begin{itemize}
\item[(a)] $A_n=\bigcup_{i=1}^ {mk(n)} \gamma_i^ {n} $, where
$\gamma_i^ n=\widehat{a_i^ n b_i^ n}$ is closed and $\gamma_i^ n
\cap \gamma_j^ n=\emptyset$ for all $i\neq j$ and $i,j=1,...,
mk(n)$. Moreover $k(n)=2n-1$ and $\gamma_{i+mj}^
n=R_\frac{j}{m} \gamma_i^ n$ for $j=1,...,m-1$ and
$i=1,...,k(n)$.

\item[(b)] $\bigcup_{j=1}^n\{a_i^ {j-1}, b_i^
{j-1}\}_{i=1}^{mk(j-1)} \subset A_n$. That is, every extreme point
of each arc $\gamma_i^{k}$ of $A_k$  belongs to $A_n$, for
$k=1,..,n-1$.

\item[(c)] $A_{n} \subset A_{n-1}$.

\item[(d)] $R_{\frac{1}{m}} A_n=A_n$.

\item[(e)] The correspondence which associates to each point $\bar
\varphi_N=\overline{\varphi+N\tau}$ the arc $\beta_N$, for all
$\abs{N}\leq n-1$, preserves the cycle order. Equivalently, if
$\bar \varphi_{N_1}\prec \bar \varphi_{N_2}\prec \bar
\varphi_{N_3}$, then $\beta_{N_1}\prec \beta_{N_2}\prec
\beta_{N_3}$  for all $\abs{N_i}\leq n-1$.

\end{itemize}
\medskip

Then, the set $$C=\bigcap_{n\in \N} A_n=\T\setminus \bigcup_{n\in
\N}(\bigcup_{k=0}^{m-1} \beta_n^k \cup \beta_{-n}^k)$$ is a Cantor
set such that $R_{\frac{1}{m}}C=C$ and the set of accessible
points of $C$ is the union of the two extremes of all components
$\gamma_i^n$ for all $n\in \N$.
\medskip

If we consider the functions $$\mathcal{P}_{*}: \bigcup_{n\in \Z}
(\bigcup_{k=0}^{m-1} \dot \beta_n^k ) \longrightarrow
\bigcup_{k=0}^{m-1} \{\overline{\varphi^k+n\tau} : n\in \Z\}$$
such that $\mathcal{P}_{*}(\dot \beta_n^k)=\bar \varphi_n^k$ for
all $k=0,...,m-1$, $n\in \Z$, where $\bar
\varphi_n^k=\overline{\varphi^k+n\tau}$. Analogously to case
$m=2$, $\mathcal{P}_{*}$ is well defined and for all
$k=0,...,m-1$,
$$\mathcal{P}_{*}(\beta_n^k)=\mathcal{P}_{*}(\beta_n^0)+\frac{k}{m}.$$

By Lemma \ref{lemHerman} we can extend $\mathcal{P}_{*}$ to a
function $\mathcal{P}: \T \to \T$ which is the associated Cantor
function to $C$ such that $\mathcal{P}(A)=D_m$ and it is
$\Z_m-$equivariant.
\medskip

Now consider the bijection
$$f_{*}:\bigcup_{n\in \Z} (\bigcup_{k=0}^{m-1} \dot \beta_n^k )
\longrightarrow \bigcup_{n\in \Z} (\bigcup_{k=0}^{m-1} \dot
\beta_n^k )$$ verifying:

\begin{itemize}
 \item For each $n\in \Z$ and $k=0,...,m-1$, $f_{*}(a_n^k)=a_{n+1}^k$ and
$f_{*}(b_n^k)=b_{n+1}^k$, where $a_n^k, b_n^k$ and $a_{n+1}^k, b_{n+1}^k$ are the two
extremes of $\beta_n^k$ and $\beta_{n+1}^k$, respectively.
 \item $f_{*}\circ R_{\frac{1}{m}}=R_{\frac{1}{m}} \circ f_{*}$.
\end{itemize}

Analogously as in Proposition \ref{propDenjoy2} $f_{*}$ is well
defined. Moreover, by Lemma \ref{lemHerman}, we can extend $f_{*}$ to a function $f:\T \to \T$ such that:

\begin{itemize} \item[(a)] $f$ is $\Z_m-$equivariant.
\item[(b)] $\rho(f)=\bar \tau$. \item[(c)] $f$ has a minimal
Cantor set.
\end{itemize}
\end{proof}

Observe that the construction of the $\Z_m-$equivariant Denjoy depend on the given point $\bar \varphi$. As different points in $\T$ define different Cantor sets we have that the construction of the $\Z_m-$equivariant Denjoy maps of Theorem \ref{teoDenjoyZm} is not unique.
%


%

\section{Main result} \label{secZnR2}

This work is motivated by the study of the global dynamics of an
equivariant planar map with an asymptotically stable fixed point.
Authors in \cite{ACL1} prove that the symmetry forces the
existence of a global attractor in all cases except $\Z_m$.
Moreover, they give in \cite{ACL2} a family of $\Z_m-$equivariant
homeomorphisms with an asymptotically stable fixed point and rotation
number $1/m$. That paper raises the question of whether there are
some relationship between the rotation number and the order of the
group $\Z_m$. In this section we construct $\Z_m-$equivariant
homeomorphisms with an asymptotically fixed point and irrational
rotation number. Consequently, there are no linkages between
rotation number and the order of the group. Hence, symmetry
properties does not give any extra information about global
dynamics in the case $\Z_m$.
\medskip

A dissipative and orientation preserving homeomorphism of the
plane $h:\R^2 \to \R^2$ is said to be an $U-$admissible (or
admissible) map provided that has an asymptotically stable fixed
point with proper and unbounded basis of attraction $U\in \R^2$.
Note that the proper condition follows when the fixed point is not
a global attractor. We can obtain automatically the unboundedness
condition if we suppose that $h$ is area contracting.
\medskip

It is well studied in \cite{pommerenke} the Theory of Prime Ends
which states that an admissible map $h$ induces an orientation
preserving homeomorphism $h^{*}:\mathbb{P} \to \mathbb{P}$ in the
space of prime ends. This topological space is homeomorphic to the
cycle, that is $\mathbb{P} \simeq \T$, and hence the rotation
number of $h^{*}$ is well defined, say $\bar \rho \in \T$. The
rotation number for an $U-$admissible orientation preserving maps
is defined by $\rho(h,U)=\bar \rho$.
\medskip

In \cite{irrationalRotNumer} the authors prove that given an
irrational number $\tau\notin \Q$ and a Denjoy homeomorphism $f:\T
\to \T$, there exists an $U-$admissible map $h:\R^2 \to \R^2$ with
rotation number $\rho(h,U)=\bar \tau$. That motivate us to prove
the existence of  $Z_m-$equivariant homeomorphisms of the plane
which induce a Denjoy map in the circle of prime ends using the
construction given in Section \ref{secZm}.
\medskip



%

\begin{proposition}[Corbato, Ortega and Ruiz del Portal, \cite{irrationalRotNumer}] \label{propRafa} Given a $\tau \in (0,1)\setminus \Q$ and a Denjoy map $f$, there exists an admissible map with rotation number $\rho(h,U)=\bar \tau$ and such that $h^{*}$ is topologically conjugate to $f$.
\end{proposition}

\begin{theorem} \label{teoZmDenjoyR2}Given an irrational number $\tau\notin \Q$, there exists a $\Z_m-$equivariant and admissible map in $\R^2$ with rotation number $\bar \tau \in \T$  and such that induces a Denjoy map in the circle of prime ends which is also $\Z_m-$equivariant.
\end{theorem}

\begin{proof} Let  $\tau \notin \Q$ an irrational number and $\bar \varphi \in \T$ be a point in the circle.
By Theorem \ref{teoDenjoyZm}, it is possible to construct a
Denjoy map $f:\T \to \T$ which is $\Z_m-$equivariant and has
minimal Cantor set $C$ verifying that $R_{\frac{2k\pi}{m} } C=C$,
for all $k=0,...,m-1$.
\medskip
%

By Proposition \ref{propRafa}, there exists an admissible map with
rotation number $\rho(h,U)=\bar \tau$ and such that the induced map in the space of prime ends $h^{*}$ is
topologically conjugate to $f$. Authors in \cite{irrationalRotNumer} define the homeomorphism $h$ in polar
coordinates by:

$$h: \quad \theta_1=f(\theta), \quad \rho_1=R(\theta, \rho)$$ where $R: \T \times [0, +\infty) \to [0,+\infty)$ such that
$$R (\bar \theta, \rho) = \left\{\begin{array}{l l l} \frac{1}{2} \rho & \text{if } \rho\leq \frac{1}{2} \\   (\frac{3}{4}-\Pi(\bar \theta))(2\rho-1)+ \frac{1}{4} & \text{if } \frac{1}{2}< \rho \leq 1 \\ \frac{1}{2}\rho +\frac{1}{2}-\Pi(\bar \theta) & \text{if } \rho >1 \end{array} \right. $$ and $\Pi: \T \to \R$ is such that $$\Pi (\bar \theta) = \left\{\begin{array}{ll} 0 & \text{if } \bar{\theta} \in C\\   \frac{1}{k(\abs{n}+1)} \frac{dist_\T(\bar \theta, C)}{length(\beta_{n}^k)} & \text{if } \bar{\theta} \in \beta_{n}^k \end{array} \right. $$
\medskip

We claim that $h$ is $\Z_m-$equivariant. Since  $f$ is
$\Z_m-$equivariant, we obtain that $$f(\bar \theta + \frac{k}{m})=
f(\bar \theta)+ \frac{k}{m}, \quad \forall \bar \theta \in \T$$ so
we only need to verify that $$R(\bar \theta + \frac{k}{m} )=
R(\bar \theta)$$ or equivalently, $$\Pi(\bar \theta + \frac{k}{m}
)= \Pi(\bar \theta)$$

Indeed, if $\bar \theta + \frac{k}{m} \in C$ then $\bar \theta \in C$ and
$\Pi(\bar \theta + \frac{k}{m})= \Pi(\bar \theta)$. Otherwise,
there exists an arc $\beta_n^j$ such that $\bar \theta +
\frac{k}{m} \in \beta_n^j$. In this case, $\bar \theta \in
\beta_n^{j-k}$ and $dist_\R(\bar \theta, C)= dist_\R(\bar \theta +
\frac{k}{m}, C)$, so $\Pi(\bar \theta + \frac{k}{m})= \Pi(\bar
\theta)$ and $h$ is $\Z_m-$equivariant.
\end{proof}

Observe that the $\Z_m-$equivariant and admissible map constructed
in Theorem \ref{teoZmDenjoyR2} depends on the initial point $\bar
\varphi$, so the constructed map is not unique.

%
%
%
%
%
%
\vspace{1cm}
\noindent {\bf Acknowledgements.} I wish to thank Prof. Rafael Ortega for fruitful conversations, 
several of which took place at the University of Granada, Spain, whose hospitality is gratefully acknowledged. I also
thank him for his helpful comments on preliminary versions of this paper.

\end{document}